\documentclass{cmsart}
\usepackage[utf8]{inputenc}
\usepackage{latexsym}
\usepackage{amsfonts}
\usepackage{amsthm}
\usepackage{amsmath}
\usepackage{enumerate}
\usepackage{graphicx,xcolor}

\newtheorem{theorem}{Theorem}[section]

\newtheorem{lemma}[theorem]{Lemma}
\newtheorem{definition}[theorem]{Definition}

\theoremstyle{remark}
\newtheorem*{remark}{Remark}
\newtheorem{example}[theorem]{Example}

\def\Z{\mathbb{Z}}
\def\Q{\mathbb{Q}}

\def\th{{^{\rm th}}}

 \def\D{{\partial}}
 \def\BCH{{\rm BCH}}
 
 \def\Aut{{\rm Aut}}
 
 \def\ad{{\rm ad}}
 \def\half{{\frac{1}{2}}}

 \def\lra{{\longrightarrow}}
 \def\takes{{\colon}}

\newcommand{\ig}[2]{{\color{lightgray} #1}{\color{red} #2}}

\begin{document}

\title{Explicit symmetric DGLA models of 3-cells}

\author{Itay Griniasty}
\address{Laboratory of Atomic and Solid State Physics, Cornell University, Ithaca, New York 14853}
\email{ig324@cornell.edu}

\author{Ruth Lawrence}
\address{Einstein Institute of Mathematics, Hebrew University of Jerusalem, Jerusalem 91904}
\email{ruthel.naimark@mail.huji.ac.il}

\begin{abstract}
We give explicit formulae for differential graded Lie algebra (DGLA) models of 3-cells. In particular, for a cube and an $n$-faceted banana-shaped 3-cell with two vertices, $n$ edges each joining those two vertices and $n$ bi-gon 2-cells, we construct a model symmetric under the geometric symmetries of the cell fixing two antipodal vertices. The cube model is to be used in forthcoming work for discrete analogues  of differential geometry on cubulated manifolds.
\end{abstract}


 \maketitle

\section{Introduction}

The discretization of differential equations is a necessity of life -- most differential equations cannot be solved analytically, and a discrete numeric solution is the choice approach for many applications. However, by their nature, discrete differences lose ``associativity" preserved in continuous derivatives \cite{GGL}. A recent programme attempts to cure this problem by constructing an analogue to differential geometry with an associative (non-commutative) infinity structure \cite{LRS}. The starting point is to associate to cell complexes a differential graded Lie algebra (DGLA).

For a regular CW complex $X$, it is possible to associate a DGLA model $A=A(X)$ over $\mathbb{Q}$ satisfying the following conditions
 \begin{itemize}
 \item as a Lie algebra, $A(X)$ is freely generated by a set of generators, one for each cell in $X$ and whose grading is one less than the geometric degree of the cell;
 \item vertices (that is $0$-cells) in $X$ give rise to generators $a$ which satisfy the Maurer-Cartan equation $\partial a+\frac12[a,a]=0$ (a flatness condition);
 \item for a cell $x$ in $X$, the part of $\D{x}$ without Lie brackets is the geometric boundary $\D_0x$ (where an orientation must be fixed on each cell);
 \item (locality) for a cell $x$ in $X$, $\D{x}$ lies in the Lie algebra generated by the generators of $A(X)$ associated with cells of the closure $\bar{x}$.
 \end{itemize}

 The existence and general construction of such a model was demonstrated by Sullivan in the Appendix to \cite{TZ}; however the procedure given there is an iterative algorithm.  By \cite{B}, there exist consistent (even symmetric) towers of models of simplices, and such towers are unique up to (exact) DGLA isomorphism. The explicit (unique) model associated with the interval was found in \cite{LS}. Explicit symmetric models of two-dimensional cells were demonstrated in previous work for bi-gons (see \cite{G}, \cite{GGL}) and arbitrary $n$-gons (see \cite{GL}), the main intermediate step being the construction of a `symmetric point' in the model of the boundary of the cell, invariant under the full symmetries of the cell. In this paper we explicitly construct for the first time DGLA models of $3$-cells, in particular for a banana-shaped cell (see Theorem 1) and for a cube (see \S\ref{sec:polyhedra}), that are invariant under symmetries  fixing a major diagonal.

In \S\ref{sec:generalproperties} we give a collection of general facts about DGLAs and models of cell complexes (some reproduced from \cite{LS}, \cite{GL}), including defining the notion of a point (solution of MC), of a particular $k$-cell ($k>1$) being localised at a point in a model, how to `twist' a model so as to move the points of localisation of cells as well as the universal averages of \cite{LSi}. In \S\ref{sec:banana}, we construct a symmetric model of the $n$-faceted banana in which the main cell is localised at a symmetric central point, by first constructing a  model localised at one of the vertices of the banana and then twisting it. In \S\ref{sec:polyhedra}, we derive a model of an arbitrary polyhedral 3-cell and give the example of the cube where the induced model is invariant under those symmetries of the cube fixing a main diagonal; this model will be used in \cite{LRS}.

\section{General properties of DGLAs and DGLA models}
\label{sec:generalproperties}

\subsection{GENERAL DGLAs} For simplicity we will work over $k=\mathbb{Q}$, though the discussion also holds for any field of characteristic zero. Recall that a DGLA over $k$ is a vector space
$A$ over $k$ with $\Z$-grading $A=\oplus_{n\in\Z}A_n$ along
with a bilinear map $[.,.]\takes{}A\times{}A\lra{}A$ (bracket, respecting the grading)  and a linear map $\D\takes{}A\lra{}A$ (differential, grading shift $-1$) for which $\D^2=0$
while
 \begin{description}
 \item[symmetry of bracket] $[b,a]=-(-1)^{|a||b|}[a,b]$
 \item[Jacobi identity] $(-1)^{|a||b|}[[b,c],a]+(-1)^{|b||c|}[[c,a],b]+(-1)^{|c||a|}[[a,b],c]=0$
 \item[Leibniz rule] $\D[a,b]=[\D{a},b]+(-1)^{|a|}[a,\D{b}]$
 \end{description}
 for all homogeneous $a,b,c\in{}A$. Defining the {\sl adjoint} action of $A$ on itself by $\ad_e(a)=[e,a]$, the operator $\ad_e:A\longrightarrow{}A$ has grading shift $|e|$, for homogeneous $e\in{}A$. The Jacobi identity and Leibnitz rule can now be reformulated as operator equalities
 \begin{description}
 \item[Jacobi identity] $\ad_{[a,b]}=[\ad_a,\ad_b]$;
 \item[Leibniz rule] $\ad_{\D{a}}=[\D,\ad_a]$;
 \end{description}
in terms of the {graded} operator commutator, $[A,B]\equiv{}A\circ{}B-(-1)^{|A||B|}B\circ{}A$. Since the relations all preserve the number of brackets, it is meaningful to define an additional grading by the number of (Lie) brackets; in particular, for $x\in{}A$, let $x^{[m]}$ denote the part of $x$ containing precisely $m$ brackets.

\subsection{POINTS AND LOCALISATION} An element $a\in{}A_{-1}$ is called a {\it point} (or said to be {\it flat}) in the model, if it satisfies the Maurer-Cartan equation $\D{a}+{\tfrac12}[a,a]=0$.  For any point $a\in{}A_{-1}$, define the {\sl twisted differential} $\D_a$ by $\D_a\equiv\D+\ad_a$; the fact that $\D_a^2=0$ is guaranteed by the Maurer-Cartan condition. By the {\sl  localisation} of $A$ to a point $a$, denoted $A(a)$, we will mean the DGLA which as a graded Lie algebra is
$$\left(\ker\D_a|_{A_0}\right)\oplus\bigoplus_{n>0}A_n$$
with the induced bracket from $A$ and the differential $\D_a$. This contains only non-negative gradings. Leibniz guarantees that $\ker\D_a|_{A_0}$ is closed under Lie bracket.

\subsection{EDGES AND FLOWS} Any element $e\in{}A_0$ defines a {\it flow} on $A$ by
 $$\frac{dx}{dt}=\D{e}-\ad_e(x)\quad{\rm on}\quad A_{-1}\>,\qquad
 \frac{dx}{dt}=-\ad_e(x)\quad{\rm on}\quad A_{\not=-1}\>,\qquad (1)$$
This flow is called the {\it flow by} $e$, and preserves the grading. (To define this rigorously, one may work in a space quotiented by all expressions involving $N+1$ Lie brackets, as in \cite{LS}, effectively truncating to the space of linear combinations of terms involving at most $N$ Lie brackets, whose coefficients are polynomials in $t$ with rational coefficients. Then one considers the tower of spaces as $N$ increases. Equivalently, one may choose a basis for the finite-dimensional space of expressions involving exactly $N$ Lie brackets and then allowed expressions are formal combinations of these basis elements, over all $N$, with coefficients which are polynomials in $t$. While we talk about functions of $t$ and their derivatives, these are well-defined for rational $t$, with derivatives being well-defined since all the coefficients are polynomial functions of $t$.)

\begin{lemma} For any $e\in{}A_0$, the flow by $e$ in grading $-1$ preserves flatness. That is, if $x(t)\in{}A_{-1}$ satisfies (1) with initial condition $x(0)$ satisfying the Maurer-Cartan condition, then at any (rational) time $t$, also $x(t)$ satisfies Maurer-Cartan.
\end{lemma}
\begin{proof} As in the proof of Theorem 1 in \cite{LS}, consider the curvature $f(t)\in{}A_{-2}$ defined by $f\equiv{}\D{x}+{\frac12}[x,x]$. It satisfies
\begin{align*}
\frac{df}{dt}&=\D\frac{dx}{dt}+\left[x,\frac{dx}{dt}\right]=-\D(\ad_ex)+[x,\D{e}]-[x,\ad_e(x)]\\
&=-\D\circ\ad_e(x)+\ad_{\D{e}}(x)+(\ad_x)^2e=-\ad_e\circ\D(x)+{\tfrac12}\ad_{[x,x]}e=-\ad_ef\>,
\end{align*}
a first order homogeneous linear ode for $f(t)$ with initial condition $f(0)=0$, since $x(0)$ satisfies the Maurer-Cartan condition. Thus $f(t)=0$ for all $t$, as required.
\end{proof}

\noindent Linearity of the differential equations (1) in $e$, ensures that flowing by $e$ for time $t$ is equivalent to flowing by $te$ for a unit time. Denote the result of flowing by $e$ from $a\in{}A_{-1}$ for unit time, by $u_e(a)$, so that the solution of the first equation in (1) is $x(t)=u_{te}(x(0))$. Explicitly
$$u_e(a)=e^{-\ad_e}a+\frac{1-e^{-\ad_e}}{\ad_e}\D{e}
=a+\D{e}-[e,a+\tfrac12\D{e}]+(\geq2\ \hbox{brackets})$$
where the meaning of  the operator quotient is the series $\sum\limits_{n=0}^\infty\frac{(-1)^n}{(n+1)!}(\ad_e)^n(\D{e})$.

\begin{lemma} For a point $a$, the condition that $u_e(a)=a$ is equivalent to $\D_ae=0$, that is $e\in{}A(a)_{0}$ ($e$ is localised at $a$). This is a linear condition on $e$ and therefore in this case the flow by $e$ fixes $a$ at all time (not only after unit time).
\end{lemma}

\begin{lemma}(see \cite{GGL}, Lemma 2.2) If $e$ flows from a point $a$ to a point $b$ in unit time, then $\D_b\circ\exp(-\ad_e)=\exp(-\ad_e)\circ\D_a$ so that $\exp(-\ad_e)$ intertwines the localisation $A(a)$ to the localisation $A(b)$.
\end{lemma}

\begin{example} The unique DGLA model, $A(I)$, of an interval has three generators; $a$, $b$ of grading $-1$ (the endpoints) and $e$ of grading $0$ (the 1-cell). The differential is given by the condition $u_e(a)=b$ (see \cite{LS}). Explicitly
$$\D{e}=(\ad_e)b+\sum_{i=0}^\infty{\frac{B_i}{i!}}(\ad_e)^i(b-a)=\frac{E}{1-e^E}a+\frac{E}{1-e^{-E}}b=b-a+\tfrac{E}2(a+b)+\cdots$$
where $B_i$ denotes the $i\th$ Bernoulli number defined as
coefficients in the expansion
$\frac{x}{e^x-1}=\sum\limits_{n=0}^\infty{}B_n \frac{x^n}{n!}$, $E\equiv\ad_e$ and the expressions in $E$ are considered as formal power series.
\end{example}

\begin{example} In any DGLA model $A(X)$ of a regular cell complex $X$, for any 1-cell $e$ in $X$ with endpoints $a$, $b$, there is a natural DGLA homomorphism $A(I)\longrightarrow{}A(X)$, while $u_e(a)=b$.
\end{example}

\noindent Denote by $\BCH(x,y)$ the {\sl Baker-Campbell-Hausdorff formula} for the element of the free Lie algebra (over $\Q$) in two variables $x$, $y$ such that as formal series
$\exp(x).\exp(y)=\exp\BCH(x,y)$ (see \cite{E} for a short proof of existence and \cite{D} for a computational formula). Here we collect some elementary properties which follow from the definition, Jacobi and uniqueness of BCH as a free Lie algebra element.

\begin{lemma}
 \begin{itemize}
 \item[(a)] The first few terms of $\BCH(x,y)$ are
     \begin{align*}
    \BCH(x,y)=x+y&+\frac{1}{2}[x,y]+\frac{1}{12}(X^2y+Y^2x)-\frac{1}{24}XYXy\!+\!\cdots
     \end{align*}
     where $X,Y$ denote $\ad_x,\ad_y$.
 \item[(b)]  $\BCH(\ad_x,\ad_y)=\ad_{\BCH(x,y)}$.
 \item[(c)] BCH is associative, that is $\BCH\big(\BCH(x,y),z\big)=\BCH\big(x,\BCH(y,z)\big)$ for any symbols $x,y,z$. The iterated BCH of $n$ symbols $x_1,\ldots{}x_n\in{}A$ will be written $\BCH(x_1,\ldots,x_n)$.
\end{itemize}
\end{lemma}

\begin{lemma} There is a homomorphism from the group $A_0$ considered with operation $\BCH$, to the group $\Aut(A)$, defined by mapping $e\in{}A_0$ to the flow (in unit time) as defined on all gradings in $A$ by equations (1).
\end{lemma}
\begin{proof}
By \cite{LS} Lemma 3,  and the explicit formula given for $u_e(a)$ above, it follows that
$$u_{e_2}\left(u_{e_1}(a)\right)=u_{\BCH(e_1,e_2)}(a)\>,$$
for any $a\in{}A_{-1}$. Thus, on elements of grading $-1$, a flow by $e_1$ for unit time followed by a flow by $e_2$ for unit time is equivalent to a flow by $\BCH(e_1,e_2)$ for unit time. Note that the flow for unit time by $e$ acting on $A_n$ for $n\not=-1$, is just the exponential operator $\exp(-\ad_e)$ for which it is immediate that
$\exp(-\ad_{e_2})\circ\exp(-\ad_{e_1})=\exp(-\ad_{BCH(e_1,e_2)})$.\end{proof}

\begin{definition} By a piecewise linear path $\gamma$ in $A$, is meant a sequence of points $a_i\in{}A_{-1}$ ($0\leq{}i\leq{}m$)  along with elements $e_i\in{}A_0$ ($1\leq{}i\leq{}m$), called edges, which are such that the edges define flows between the respective points, that is $u_{e_i}(a_{i-1})=a_i$ for all $1\leq{}i\leq{}m$.  For such a path, we denote by $\BCH(\gamma)\in{}A_0$ the iterated $BCH$ of the edges, $\BCH(\gamma)\equiv\BCH(e_1,\ldots,e_m)$. A piecewise linear path in $A$ is called a loop if its initial and final points agree, that is $a_0=a_m$.
\end{definition}

\begin{lemma} (see \cite{B4}) If $X$ has $c$ connected components and $\{a_1,\ldots,a_c\}$ is a choice of basepoints, one in each connected component, then  the set of points in $A(X)$ is
$$\bigcup_{i=1}^c\big\{u_e(a_i)\bigm|e\in{}A_0\big\}\cup
\big\{u_e(0)\bigm|e\in{}A_0\big\}\>.$$
For each $i$, the map $\pi_i\takes{}e\mapsto{}u_e(a_i)$ is a `fibration', with fibre $\pi_i^{-1}(a_i)$ generated as a vector space by $\{\BCH(\gamma)|\gamma\in\pi_1(X,a_i)\}$, while the map $\pi_0\takes{}e\mapsto{}u_e(0)$ is injective.
\end{lemma}

\subsection{LOCALISATION OF MODELS}  As noted in \S1, $A(X)$ is not unique, but is well-defined up to (exact) DGLA isomorphism.

\begin{definition}
A point $a\in{}A^{-1}$ in a model of a regular cell complex $X$, is said to be \underline{local to a cell} $f$ in $X$ if it lies in the submodel generated by the cells in the closure $\overline{f}$ of $f$.
\end{definition}

\noindent By Lemma 2.9, equivalently, a point is local to the cell $f$ if it can be written as $u_e(a_0)$, where $a_0$ is a 0-cell in $\overline{f}$ and $e$ is a zero-graded element in the (free) Lie algebra generated by cells in $\overline{f}$.

\begin{definition}
In a DGLA model $A$ of a regular cell complex $X$, a $k$-cell $f$ (for $k>1$) will be said to be \underline{localised at the point} $a\in{}A_{-1}$ if $\D_af=\D{}f+[a,f]$ lies in the (free) Lie algebra generated by cells of the closure of the geometric boundary $\overline{\D_0f}$.
\end{definition}

\noindent Here, by abuse of notation, we have used the same symbol for the geometric $k$-cell $f$ and the element $f\in{}A_{k-1}$ in its model. By locality of the model and using freeness of the Lie algebra, we see that the point of localisation of a cell $f$ in a model (if it exists) is unique and must be local to the cell (in the sense of Definition 2.10).

\begin{remark}The explicit constructions of models of the bi-gon in \cite{GGL} and triangle in \cite{GL}, are localised (at their `centre' points). Although not all models of $X$ will be such that all cells of dimension $>1$ are localised (for example in the model of the bi-gon in \cite{G}, the main cell is not localised at any point), we will see in \S4 that such models always exist in dimensions up to three. A similar construction should work also in higher dimensions.
\end{remark}

\begin{lemma}
If $A$ is a model of a regular cell complex $X$ in which the $k$-cell $f$ ($k>1$) is localised at the point $a\in{}A_{-1}$, then for any \ig{other}{} $e\in{}A_0$ in the subalgebra generated by the 1-skeleton of $\overline{f}$, there is a variation $A'$ of the model in which the generator for $f$ is replaced by
\[f'=\exp(-\ad_e)\cdot{}f\]
and the cell $f$ is localised at the point $a'=u_e(a)$.
\end{lemma}
\begin{proof}
This is immediate from Lemma 2.3, $\D_{a'}f'=\exp(-\ad_e)\cdot\D_af$.
\end{proof}

\noindent The above lemma means that we can `twist' any model in which a cell is localised, so that the cell is localised at any point we please (which is local to the cell). This technique was used in \cite{GGL} and \cite{GL} to generate symmetric models, by writing first a (non-symmetric) model of the relevant 2-cell localised at a point on its boundary and then twisting it so that it is localised at a symmetric point. All that remained was to verify that the model obtained was indeed symmetric.

\begin{remark}
Since the 1-skeleton of $\overline{f}$ is not simply connected, the construction given by Lemma 2.12 of a model $A'$ in which $f$ is localised at another point $a'$ local to the face $f$, will not be unique. That is, there are different $e\in{}A_0$ for which $u_e(a)=a'$. In particular, a twist by $t\BCH(\gamma)$ of a model in which $f$ is localised at $a_0$, for any $t\in\Q$ and any non-trivial loop $\gamma$ based at $a_0$ in the 1-skeleton on $\overline{f}$, will yield another (distinct) such model.
\end{remark}

\subsection{UNIVERSAL AVERAGES} In \cite{LSi}, it was shown how to construct a universal expression $\mu_n(x_1,\ldots,x_n)$ in the free Lie algebra of $x_1,\ldots,x_n$ such that
\begin{itemize}
\item[(i)] $\mu_n$ is totally symmetric in its arguments,
\item[(ii)] in any DGLA model in which $a,b$ are points with $u_{x_i}(a)=b$ for all $i=1,\ldots,n$, also $u_{\mu_n(x_1,\ldots,x_n)}(a)=b$.
\end{itemize}
It was also shown that the expansion of $\mu_n$ up to three Lie brackets is
\[
\mu_n(x_1,\ldots,x_n)=\tfrac1n\sum_ix_i-\tfrac1{12n^2}\sum_{i,j,\ i\not=j}[x_i,[x_i,x_j]]+\cdots
\]
and we call $\mu_n$ the universal average. There is a closed formula for $\mu_2$, found in \cite{GGL},
\[\mu_2(x,y)=\BCH(x,\tfrac12\BCH(-x,y))\]

\section{The banana 3-cell}
\label{sec:banana}
Let $X_n$ be the $n$-faceted banana, with two 0-cells, $n$ 1-cells, $n$ bi-gon 2-cells and one 3-cell. The corresponding model will contain Lie algebra generators corresponding to each cell; denote them by $a,b$ (grading $-1$), $e_i$ ($1\leq{}i\leq{}n$, grading 0), $f_i$ ($1\leq{}i\leq{}n$, grading 1) and $h$ (grading 2) respectively. Here the orientation on the 2-cell $f_i$ is chosen so that its geometric boundary is $\D_0f_i=e_{i}-e_{i+1}$ where $e_{n+1}\equiv{}e_1$ (indices modulo $n$). The geometric boundary of $h$ is $\D_0h=f_1+\cdots+f_n$.
\[\includegraphics[width=.4\textwidth]{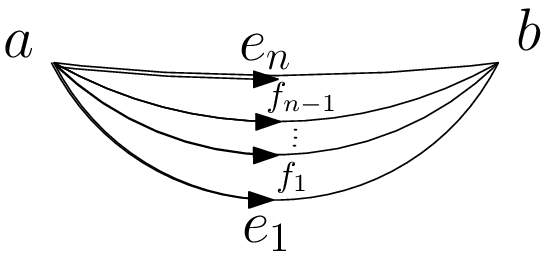}\]

\subsection{Symmetries} The geometric symmetry group of the banana is $D_n\times{}\Z_2$. The dihedral group is generated by the rotation by $\frac{2\pi}n$ around the axis $ab$,
\[\tau: e_i\longmapsto{}e_{i+1},\quad{}f_i\longmapsto{}f_{i+1},\quad{}a,b,h\ \hbox{fixed}\]
and the reflection in the plane containing $e_n$ and the central axis of the banana,
\[\sigma: a,b\ \hbox{fixed},\quad{}e_i\longmapsto{}e_{n-i},
\quad{}f_i\longmapsto{}-f_{n-i-1},\quad{}h\longmapsto-h\]
while the further $\Z_2$ factor is generated by reflection in the plane equidistant from $a$ and $b$,
\[\iota:\ a\longleftrightarrow{}b,\quad{}e_i\longmapsto-e_i,\quad{}
f_i\longmapsto-f_i,\quad{}h\longmapsto-h\]
which commutes with both $\sigma$ and $\tau$. For some purposes we will restrict to the subgroup $D_n=\langle\sigma,\tau\rangle$ of the symmetry group fixing the vertices.

\subsection{0-,1-,2-cells} By the Leibniz rule, the differential $\D$ is determined by its values on generators. On vertices, $\D$ is fixed by the Maurer-Cartan condition, namely
\[\D{a}=-\half[a,a],\qquad \D{b}=-\half[b,b]\>.\eqno{(1)}\]
On $1$-cells, $\D$ is also unique (see Example 2.4),
\[\D{e_i}=\frac{E_i}{1-e^{E_i}}a+\frac{E_i}{1-e^{-E_i}}b
=b-a+\tfrac12[a+b,e_i]+(\geq2\ \hbox{brackets})\eqno{(2)}\]
where $E_i=\ad_{e_i}$. The faces $f_i$ are bi-gons, and we use the symmetric model of the bi-gon from \cite{GGL} in which
\[
\D{}f_i=\BCH(-\tfrac12v_i,e_i,-e_{i+1},\tfrac12v_i)-[x_i,f_i]
=e_i-e_{i+1}-\tfrac12[a+b,f_i]+\cdots\eqno{(3)}
\]
so that $f_i$ is localised at its centre
\[x_i=u_{\frac12v_i}(a)
=\tfrac12(a+b)+\frac1{16}[e_i+e_{i+1},b-a]+(\geq2\ \hbox{brackets})
\] where $v_i$ is the centreline from $a$ to $b$ given by the average
\[
v_i=\mu_2(e_i,e_{i+1})=\BCH(e_i,\tfrac12\BCH(-e_i,e_{i+1}))
=\tfrac12(e_i+e_{i+1})+(\geq2\ \hbox{brackets})
\]
in terms of the universal average $\mu_2$ of \cite{LSi}.
\[\includegraphics[width=.3\textwidth]{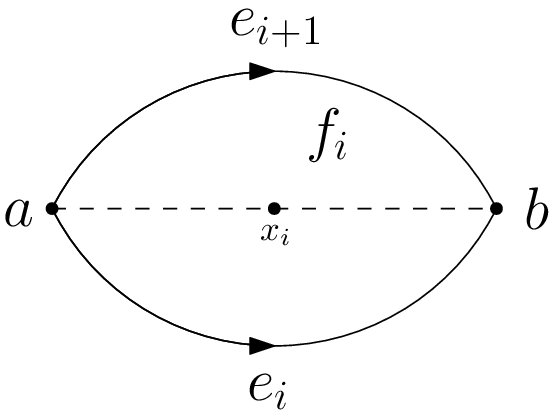}\]

\subsection{Central point of banana} A totally symmetric point should be at the centre of the banana,
\[x=u_{\frac12v}(a)
=\tfrac12(a+b)+\tfrac1{8n}\sum\limits_{i=1}^nE_i(b-a)+\cdots\]
which is half way along a central diagonal of the banana from $a$ to $b$, a path in the direction $v=\mu_n(e_1,\ldots,e_n)$, the universal average of $e_1,\ldots,e_n$. The fact that $x$ is invariant under $\sigma$, $\tau$ and $\iota$ follows from the total symmetry of $\mu_n$ along with the fact that (\cite{LSi}, Lemma 4.3)
\[\mu_n(-e_1,\ldots,-e_n)=-\mu_n(e_1,\ldots,e_n)\]
The only freedom remaining in the model is the boundary of the 3-cell, $\D{h}\in{}A_1$, which, to give a valid model of $X_n$ must be such that $\D^2(h)=0$, while $(\D{h})^{[0]}$  must coincide with the topological boundary $\D_0h=f_1+\cdots+f_n$. The purpose of this section is to give a formula for $\D{h}$ which is invariant under the $D_n$ action of the symmetries of the banana fixing the vertices and in which $h$ is localised at a central point. In order to do this, we will first construct a model in which all the 2-cells and the three-cell are localised at $a$, and then twist using Lemma 2.12 to localise at a symmetric point.

\subsection{Model localised at $a$} Since $f_i$ is localised at $x_i$, thus $f'_i=\exp(\frac12\ad_{v_i})\cdot{}f_i$ is localised at $a$, in particular
\[
\D_a{}f'_i=\BCH(e_i,-e_{i+1})\eqno{(4)}
\]
Similarly, set $h'=\exp(\frac12\ad_{v})\cdot{}h$; by Lemma 2.12, a model in which $h$ is localised at $x$ will have $h'$ localised at $a$.

\begin{lemma} A DGLA model of $X_n$ is defined by generators $a,b,e_i,f'_i,h'$, with the differential defined by (1), (2), (4) along with
\[\D_ah'=\sum\limits_{i=1}^nP_i(\BCH(E_1,-E_2),\ldots,\BCH(E_{n-1},-E_n))\cdot{}f'_i
\eqno{(5)}
\]
for any polynomials $P_i$ in $(n-1)$ non-commuting variables whose initial term is $1$ and  which satisfy the identity
\[
\sum_{i=1}^{n-1}P_i(\ad_{x_1},\ldots,\ad_{x_{n-1}})\cdot{}x_i
=P_n(\ad_{x_1},\ldots,\ad_{x_{n-1}})\cdot\BCH(x_1,\ldots,x_{n-1})
\eqno{(6)}
\]
In this model, the 3-cell and 2-cells are all localised at $a$.
\end{lemma}
\begin{proof}
It is apparent that the initial term of $\D_ah'$ as defined by (5) is $\sum\limits_{i=1}^nf'_i$ as required. It remains only to verify that $\D_a^2h'=0$, that is, that the RHS of (5) defines an element of $\ker\D_a$. For this, observe that since $u_{e_i}(a)=b$, thus by Lemma 2.7, $u_{\BCH(e_i,-e_{i+1})}(a)=a$ and so by Lemma 2.2, $\BCH(e_i,-e_{i+1})\in\ker\D_a$. However, for $y_1,\ldots,y_r\in\ker\D_a$ of degree 0,
\[\D_a[y_1,[y_2,\cdots,[y_r,w]\cdots]]=[y_1,[y_2,\cdots,[y_r,\D{}w]\cdots]]\]
and so for any non-commuting polynomial, $P$, \ig{also}{} $P(\ad_{y_1},\ldots,\ad_{y_r})$ commutes with $\D_a$. By Lemma 2.6(b), $\BCH(\ad_x,\ad_y)=\ad_{\BCH(x,y)}$ and it follows that
\begin{align*}
\D_a&\sum\limits_{i=1}^nP_i(\BCH(E_1,-E_2),\ldots,\BCH(E_{n-1},-E_n))\cdot{}f'_i\\
&=\sum\limits_{i=1}^nP_i(\BCH(E_1,-E_2),\ldots,\BCH(E_{n-1},-E_n))\cdot{}\D_af'_i\\
&=\sum\limits_{i=1}^nP_i(\BCH(E_1,-E_2),\ldots,\BCH(E_{n-1},-E_n))\cdot{}\BCH(e_i,-e_{i+1})\\
&=0
\end{align*}
where the second step follows from (4) and the third from (6) applied to $x_i=\BCH(e_i,-e_{i+1})$ since $-x_n=\BCH(x_1,\cdots,x_{n-1})$.
\end{proof}

To see that the previous lemma's constructions actually lead to a model of the $n$-faceted banana, it remains only to show that $P_1,\ldots,P_n$ exist with initial term 1 and satisfying the identity (6). This follows immediately from Lemma 2.6(c), even with $P_n\equiv1$. Indeed, there are many possible choices for $P_1,\ldots,P_n$.

\subsection{Model localised at central point $x$} Twisting the model of Lemma 3.1 back so that the 2-cells are localised at their centres and the 3-cell is localised at its centre $x$, we find the the differential is given by (1), (2), (3) along with
\[\D_xh=e^{-\frac12\ad_v}
\left(\sum\limits_{i=1}^nP_i(\BCH(E_1,-E_2),\ldots,\BCH(E_{n-1},-E_n))\cdot{}
e^{\frac12\ad{v_i}}f_i\right)
\eqno{(7)}
\]
The final requirement is that the model is symmetric under the action of $D_n$, that is, it is invariant under $\sigma$, $\tau$. By the geometry and total symmetry of the bi-gon model, (1), (2) and (3) will be invariant under $\sigma$, $\tau$ and $\iota$. Under $\tau$, $x,h,v$ are fixed while $e_i\longmapsto{}e_{i+1}$, $f_i\longmapsto{}f_{i+1}$, $v_i\longmapsto{}v_{i+1}$ so (7) is invariant under $\tau$ so long as
\begin{align*}
P_i&(\BCH(E_2,-E_3),\ldots,\BCH(E_n,-E_1))\\
&=P_{i+1}(\BCH(E_1,-E_2),\ldots,\BCH(E_{n-1},-E_n))
\end{align*}
which is ensured by
\[
P_{i+1}(x_1,\ldots,x_{n-1})=P_i(x_2,\ldots,x_{n-1},-\BCH(x_1,\ldots,x_{n-1}))\eqno{(8)}
\]
Under $\sigma$, $v,x$ are fixed, $e_i\mapsto{}e_{n-i}$, $f_i\mapsto{}-f_{n-i-1}$, $v_i\mapsto{}v_{n-i-1}$ while $h$ changes sign. In order that (7) be invariant under $\sigma$, it is required that
\begin{align*}
P_i&(\BCH(E_{n-1},-E_{n-2}),\ldots,\BCH(E_1,-E_n))\\
&=P_{n-i-1}(\BCH(E_1,-E_2),\ldots,\BCH(E_{n-1},-E_n))
\end{align*}
which is ensured by (8) along with
\[
P_{n-i}(x_1,\ldots,x_{n-1})=P_i(-x_{n-1},\ldots,-x_1)\eqno{(9)}
\]
Under $\iota$, the quantities $e_i$, $f_i$, $v_i$, $h$ and $v$ all change sign while $x$ is invariant, and so (7) is invariant under $\iota$ so long as
\[
\begin{array}{l}
e^{\frac12\ad_v}P_i(\BCH(-E_1,E_2),\ldots,\BCH(-E_{n-1},E_n))e^{-\frac12\ad_{v_i}}\\
=e^{-\frac12\ad_v}P_{i}(\BCH(E_1,-E_2),\ldots,\BCH(E_{n-1},-E_n))e^{\frac12\ad_{v_i}}
\end{array}\eqno{(10)}
\]
where $v_i=\mu_2(e_i,e_{i+1})$ and $v=\mu_n(e_1,\ldots,e_n)$.

In conclusion we have the following lemma.

\begin{lemma}
The DGLA with (free) Lie algebra generators $a,b,e_i,f_i,h$ and differential defined by (1),(2), (3) and (7) is a model for the $n$-faceted banana which is invariant under the geometric symmetries of the cell fixing the vertices, so long as the polynomials $P_1,\ldots,P_n$ in $(n-1)$ non-commuting variables, with initial term 1, satisfy the identities (6), (8) and (9). It will be completely invariant under the symmetries of the 3-cell if in addition (10) holds.
\end{lemma}

By Lemma 2.6(a), we can write \[BCH(x,y)=x+y+Q(X,Y)y\eqno{(11)}\] where $Q$ is a polynomial in two non-commuting variables whose lowest order terms are $\tfrac12X+\tfrac1{12}(X^2-YX)+\cdots$.
By Lemma 2.6(c), iterating (11) gives
\[\BCH(x_1,\ldots,x_n)=\sum_{i=1}^nx_i+\sum_{i=2}^nQ(\BCH(X_1,\ldots,X_{i-1}),X_i)x_i\]
so that
\[P_1=P_n=1,\quad{}P_i=1+Q(\BCH(X_1,\ldots,X_{i-1}),X_i)\ \hbox{for $1<i<n$}\eqno{(12)}\]
is a solution of (6). Note that (6) is a linear condition on $(P_1,\ldots,P_n)$ which is invariant under the action of cyclic permutation of the $P_i$ while cyclically permuting the $X_i$ (with $X_n=-\BCH(X_1,\ldots,X_{n-1})$), as well as reversing the order of the $P_i$ while reversing the order and signs of the $x_i$. That is, (6) is invariant under a dihedral group action and so there exists an invariant solution (which will satisfy (8) and (9)) given by averaging the orbit of the solution (12) under the just described dihedral action. The result is that
\begin{align*}
&P_i=1+\tfrac1{2n}\!\!\!\!\sum_{j=1+\delta_{in}}^{i-1}\!\!\!\!Q(\BCH(X_j,\ldots,X_{i-1}),X_i)
+\tfrac1{2n}\sum_{j=i+1}^{n-1}\!\!\!Q(-\BCH(X_i,\ldots,X_j),X_i)\\
&+\tfrac1{2n}\!\!\!\!\sum_{j=1+\delta_{in}}^{i-1}\!\!\!\!Q(\BCH(X_j,\ldots,X_i),-X_i)
+\tfrac1{2n}\sum_{j=i+1}^{n-1}\!\!\!Q(-\BCH(X_{i+1},\ldots,X_j),-X_i)
\end{align*}
satisfies (6), (8) and (9). The coefficient of $f'_i$ in $\D_ah'$ from (5) is now
\begin{align*}
P_i&(\BCH(E_1,-E_2),\ldots,\BCH(E_{n-1},-E_n))\\
&=1+\tfrac1{2n}\sum_{{j=1}\atop{j\not=i,i+1}}^{n}\Big(Q\big(\BCH(E_j,-E_i),\BCH(E_i,-E_{i+1})\big)\\[-18pt]
&\qquad\qquad\qquad\qquad+Q\big(\BCH(E_j,-E_{i+1}),\BCH(E_{i+1},-E_i)\big )\Big)
\end{align*}

\noindent{\bf Theorem 1}\ {\sl The DGLA with (free) Lie algebra generators $a,b,e_i,f_i,h$ and differential defined by (1), (2), (3) and
\begin{align*}
e^{\frac12\ad_v}\D_xh&=\sum_{i=1}^nf'_i+\tfrac1{2n}\sum_{{i,j=1}\atop{j\not=i,i+1}}^n
\Big(Q\big(\BCH(E_j,-E_i),\BCH(E_i,-E_{i+1})\big)\\[-18pt]
&\qquad\qquad\qquad\qquad\qquad+Q\big(\BCH(E_j,-E_{i+1}),\BCH(E_{i+1},-E_i)\big )\Big)f'_i
\end{align*}
defines a model for the $n$-faceted banana which is symmetric under the geometric symmetries of the cell fixing the vertices, where $Q$ is defined by (11), $f'_i=e^{\frac12\ad_{v_i}}f_i$, $v_i=\mu_2(e_i,e_{i+1})$, $v=\mu_n(x_1,\ldots,x_n)$ and $x=u_{\frac12v}(a)$. The 3-cell is localised at $x$ in this model and the bi-gon 2-cells are localised at their centres $x_i=u_{\frac12v_i}(a)$.}

\noindent Up to the first two non-trivial orders (in Lie brackets),
\begin{align*}
v_i&=\tfrac12(e_i+e_{i+1})-\tfrac1{48}(E_i^2e_{i+1}+E_{i+1}^2e_i)+\cdots\\
v&=\tfrac1n\sum_ie_i-\tfrac1{12n^2}\sum_{i\not-j}E_i^2e_j+\cdots\\
x&=\tfrac12(a+b)+\tfrac1{8n}\sum\limits_i[e_i,b-a]+\cdots\\ x_i&=\tfrac12(a+b)+\tfrac1{16}(E_i+E_{i+1})(b-a)+\cdots\\
\end{align*}
Meanwhile, up to second order (in Lie brackets) the differential is given by
\begin{align*}
\D{a}&=-\tfrac12[a,a],\qquad \D{b}=-\tfrac12[b,b]\\
\D{e_i}&=b-a+\tfrac12E_i(a+b)+\tfrac1{12}E_i^2(b-a)+\cdots\\
\D{}f_i&=e_i-e_{i+1}-\tfrac12[a+b,f_i]+\tfrac1{48}(E_{i+1}^2e_i-E_i^2e_{i+1})
+\tfrac1{16}[(E_i+E_{i+1})(a-b),f_i]+\cdots\\
\D{}h&=\sum_if_i-\tfrac12[a+b,h]+\tfrac1{8n}\sum_i[E_i(a-b),h]+\sum_i
\hbox{(quadratic in $E_j$'s)}f_i+\cdots
\end{align*}

\begin{remark}
The particular solution of (6),(8) and (9) constructed above does not satisfy (10) and so the model described in Theorem 1 is not symmetric under the full symmetry group $D_n\times\Z_2$, but only under the part fixing the vertices. In particular, it is not invariant under $\iota$, although such a model does exist. However this model, and the model it induces in \S4 on the cube,  have sufficient symmetry for the applications in \cite{LRS}, where cells will come with a preferred oriented main diagonal.
\end{remark}

\section{A model for an arbitrary polyhedral 3-cell}
\label{sec:polyhedra}
Suppose that $X$ is an arbitrary polyhedral 3-cell, with $n$ faces. Choose two vertices $a$ and $b$. Pick a shelling of the subdivision of the boundary, that is a choice of $n$ non-self intersecting paths $\gamma_1,\ldots,\gamma_n$ each from $a$ to $b$ along edges of the polyhedron, in such a way that for each $i=1,\ldots,n$, the paths $\gamma_i$ and $\gamma_{i+1}$ (with $\gamma_{n+1}\equiv\gamma_1$) have common initial and final segments so that the intermediate segments together (one in reverse orientation) form the geometric boundary of the $i^\th$ face, $g_i$,  of $X$. See the figure below; for each $i$,
\[\gamma_i=\alpha_i\cup\delta_i\cup\beta_i,\quad\gamma_{i+1}=\alpha_i\cup\delta'_i\cup\beta_i\]
while $\delta_i\cup-\delta'_i=\D_0\hbox{($i\th$ face)}$ with matching orientation (faces oriented outwards). Here $\delta_i$ and $\delta'_i$ share common initial and final points, say $p_i$ and $q_i$.
\[\includegraphics[width=.4\textwidth]{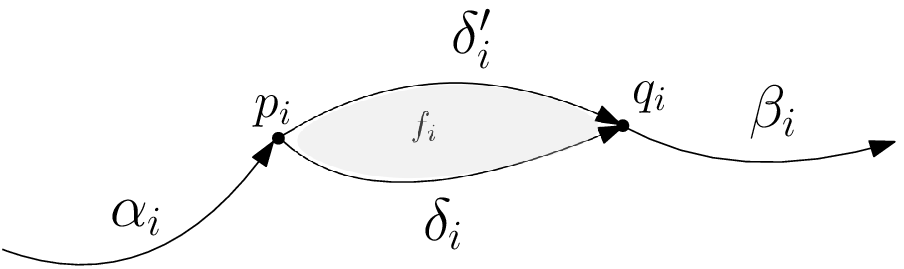}\]

\noindent To construct a model $A(X)$ of $X$, first note that 0-cells and 1-cells have a unique description. For the 2-cells, pick a basepoint (say $p_i$). In a model in which the $i^\th$ cell $f_i$ is localised at $p_i$
\[\D_{p_i}f_i=\BCH(\delta_i,-\delta'_i)\]
To determine a suitable expression for the differential of the 3-cell $h$ in $X$, we decide to localise it at $a$ and then induce $\D{}h$ from the differential (5) on the 3-cell in the $n$-faceted banana $X_n$ using the natural DGLA map $A(X_n)\longrightarrow{}A(X)$ in which
\begin{align*}
a,b&\longmapsto{}a,b\\
e_i&\longmapsto{}\BCH(\gamma_i)\\
f'_i&\longmapsto\exp(\BCH(\alpha_i))g_i\\
h'&\longmapsto{}h
\end{align*}
using the notation of $f'_i$ and $h'$ cells in $X_n$ localised at $a$ as in (4) and (5) above. Thus, one can use in $A(X)$,
\[\D_ah\!=\!\sum\limits_{i=1}^nP_i(\BCH(\ad_{\gamma_i},-\ad_{\gamma_{i+1}}),
\ldots,\BCH(\ad_{\gamma_{n-1}},-\ad_{\gamma_n}))\cdot{}\exp(\BCH(\alpha_i))g_i
\]
If it is desired to localise cells at symmetric central points in place of points on their boundary, then additional twists can be applied according to Lemma 2.12.

\begin{example}
We apply the above general argument to construct a model for the cube in which 2- and 3- cells are localised at their centres. Denote a pair of antipodal vertices by $a,b$. If these vertices and their adjacent edges are removed from the skeleton of the cube, then a hexagon remains joining the six remaining vertices; label the vertices of this hexagon $a_1,\ldots,a_6$ and the edges $f_1,\ldots,f_6$ (with $f_i$ joining $a_i$ and $a_{i+1}$ modulo 6). Denote the edges adjacent to $a$ by $e_1,e_2,e_3$ and their opposite edges (adjacent to $b$) by $\bar{e}_1,\bar{e}_2,\bar{e}_3$ respectively. Orient the edges so that those adjacent to $a$ are oriented away from $a$ and orient other edges so that parallel edges have matching orientation. The faces containing $a$ are labelled $g_1,g_2,g_3$ and their opposite faces with a bar.
\[\includegraphics[width=.4\textwidth]{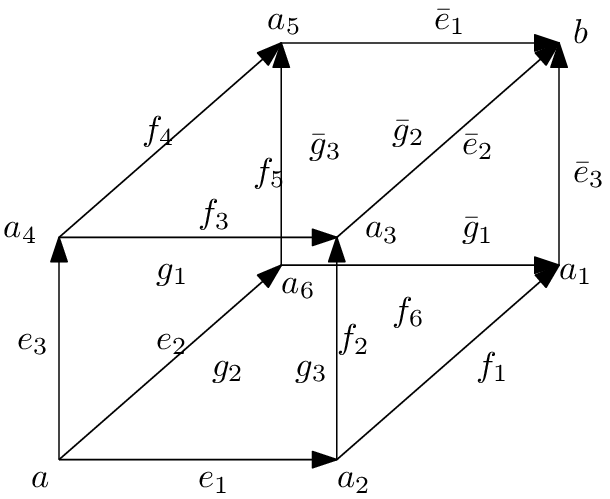}\hskip3em
\includegraphics[width=.3\textwidth]{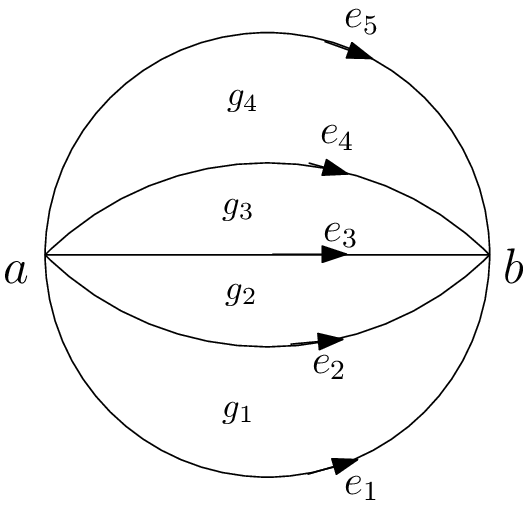}
\]
As usual, the differential on vertices $x$ has $\D{}x=-\tfrac12[x,x]$ while the differential on all edges is also rigidly determined as in Example 2.4. The edge orientations give a partial ordering on the vertices, with $a$ as minimal element and $b$ as maximal element. Each face similarly has a minimal and a maximal element which are diagonally opposite. The differential on faces of the cube is chosen so that they will be localised at their centre.
\[\includegraphics[width=.2\textwidth]{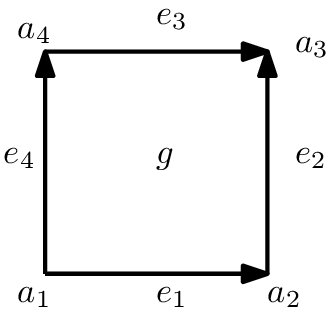}\hskip3em
\includegraphics[width=.3\textwidth]{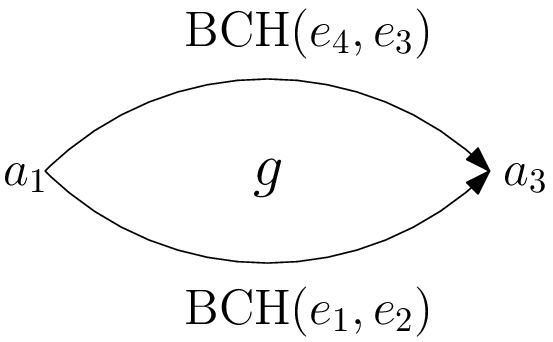}
\]
For a square as depicted, the symmetric model of the bi-gon induces \[\D_x{}g=\exp(-\tfrac12\ad_v)\BCH(e_1,e_2,-e_3,-e_4)\]
where $x=u_{v/2}(a_1)$ is the centre point defined in terms of the centre diagonal
\[v=\mu_2\big(\BCH(e_1,e_2),\BCH(e_4,e_3)\big)\]

The formula for the cube is induced from a symmetric model of a 6-faceted banana, by mapping the edges of $X_6$ to $\BCH$ of the 6 maximal
chains from $a$ to $b$ on the cube, while mapping the faces $f_i$ (based at their centres) to the conjugations of the corresponding faces
of the cube, $\exp(-V_i)g_i$, where $v_i$ is a zero graded element which flows the centre $x_i$  of the $i\th$
face of the cube to the centre of the corresponding bi-gon bounded by two maximal chains (thus changing the point at which it is localised). Here by the centre of a square, we mean the centre point as in
the formulae for $\D$ of a square above. Up to terms with only one bracket (which are all that affect the result
of $\exp(-V_i)g_i$ up to two brackets), the $v$'s for the squares $g_1$, $g_2$, $g_3$, $\bar{g}_1$, $\bar{g}_2$,
$\bar{g}_3$ are $\frac12\bar{e}_1$, $\frac12\bar{e_2}$, $\frac12\bar{e}_3$, $-\frac12e_1$,
$-\frac12e_2$, $-\frac12e_3$, respectively. That is $A(X_6)\longrightarrow{}A(\hbox{cube})$ is given on edges by
\begin{align*}
e_1&\longmapsto\BCH(e_3,f_4,\bar{e}_1)\\
e_2&\longmapsto\BCH(e_2,f_5,\bar{e}_1)\\
e_3&\longmapsto\BCH(e_2,f_6,\bar{e}_3)\\
e_4&\longmapsto\BCH(e_1,f_1,\bar{e}_3)\\
e_5&\longmapsto\BCH(e_1,f_2,\bar{e}_2)\\
e_6&\longmapsto\BCH(e_3,f_3,\bar{e}_2)
\end{align*}
and on faces by
\begin{align*}
g_1&\longmapsto-\exp(-\tfrac12\bar{E}_1){g}_1+(\geq2\hbox{ brackets})\\
g_2&\longmapsto-\exp(\tfrac12E_2)\bar{g}_2+(\geq2\hbox{ brackets})\\
g_3&\longmapsto-\exp(-\tfrac12\bar{E}_3){g}_3+(\geq2\hbox{ brackets})\\
g_4&\longmapsto\exp(\tfrac12E_1)\bar{g}_1+(\geq2\hbox{ brackets})\\
g_5&\longmapsto\exp(-\tfrac12\bar{E}_2){g}_2+(\geq2\hbox{ brackets})\\
g_6&\longmapsto\exp(\tfrac12E_3)\bar{g}_3+(\geq2\hbox{ brackets})
\end{align*}
Up to one Lie bracket in $A(X_6)$
\[\D{}h=\sum\limits_ig_i-\tfrac12[a+b,h]+(\geq2\hbox{ brackets})\]
which transforms on the cube to
\begin{align*}
\D{}h=&\bar{g}_1+g_2+\bar{g}_3-g_1-\bar{g}_2-g_3-\tfrac12[a+b,h]\\
&+\tfrac12(E_3\bar{g}_3-\bar{E}_2g_2+E_1\bar{g}_1+\bar{E}_3g_3-E_2\bar{g}_2+\bar{E}_1g_1)+(\geq2\hbox{ brackets})
\end{align*}
The model of the cube so obtained will be symmetric under symmetries which fix the diagonal $ab$.
\end{example}

\begin{remark}
Notice that the structure of the differential for 3-cells is essentially simpler than that for 2-cells, having a linear dependence on its codimension one boundary sub-cells. Thus in (5), the dependence of $\D_ah'$ upon $f_i$ is linear. This is not true one dimension lower, where the dependence in (4) of $\D_af'_i$ on its boundary cells $e_i$ and $e_{i+1}$ is highly non-linear (through $\BCH$).
\end{remark}

\begin{remark}
The functorial nature of the construction of $A(X)$ from $X$ means that under a subdivision, there should be a corresponding DGLA map. The subdivision map for the interval (or 1-skeleta) is generated by $\BCH$ (see \cite{LS}) and for this reason the differential on a 2-cell (in its simplest form) is the $\BCH$ of its boundary (see (4)). Similarly, the subdivision map for a 2-cell (say splitting a bi-gon into two bi-gons) will be linear and the difference between the two sides gives an expression for the differential on a 3-cell as in (5). This same story should continue in higher dimensions (all dimensions $>2$ will have a linear form).
\end{remark}

\subsection*{Acknowledgments}
 This research was supported in part by Grant No 2016219 from the United States-Israel Binational Science Foundation (BSF). Itay Griniasty is grateful to the Azrieli Foundation for the award of an Azrieli Fellowship.

\end{document}